\documentclass[10pt]{amsart}
\usepackage{amssymb}
\usepackage{amscd}
\usepackage{amsmath}
\usepackage{color}
\definecolor{darkspringgreen}{rgb}{0.09, 0.45, 0.27}
\usepackage[pdftex,bookmarks=false,colorlinks=true,
  linkcolor=blue,
  citecolor=darkspringgreen,debug=true,
naturalnames=true,pdfnewwindow=true]{hyperref}
\theoremstyle{plain}
\newtheorem{Thm}{Theorem}[section]
\newtheorem{Thm*}{Theorem}[section]
\newtheorem{Thm'}[Thm]{"Theorem"}
\newtheorem{Cor}[Thm]{Corollary}

\newtheorem{Conj}[Thm]{Conjecture}
\newtheorem{Prop}[Thm]{Proposition}
\newtheorem{Lem}[Thm]{Lemma}

\theoremstyle{definition}

\newtheorem{Emp}[Thm]{}

\newtheorem{Q}[Thm]{Question}

\numberwithin{equation}{section}
\newcommand{\nc}{\newcommand}
\nc{\lm}{\lambda}

\newcommand{\ov}{\overline}

\newcommand{\B}[1]{\mathbb#1}
\newcommand{\cal}[1]{\mathcal{#1}}
\newcommand{\C}[1]{\cal#1}

\newcommand{\isom}{\overset {\thicksim}{\to}}

\newcommand{\si}[2]{\text{$\Sigma^{#1}_{#2}$}}

\newcommand{\lra}{\longrightarrow}

\newcommand{\hra}{\hookrightarrow}
\newcommand{\wt}{\widetilde}

\newcommand{\g}{\mathfrak{g}}

\newcommand{\gm}{\gamma}
\newcommand{\dt}{\delta}

\newcommand{\La}{\Lambda}
\newcommand{\bs}{\backslash}

\newcommand{\un}[1]{\underline{#1}}

\newcommand{\al}{\alpha}

\newcommand{\la}{\lambda}

\newcommand{\form}[1]{(\ref{Eq:#1})}

\newcommand{\rl}[1]{Lemma~\ref{L:#1}}

\newcommand{\rp}[1]{Proposition~\ref{P:#1}}

\newcommand{\re}[1]{\ref{E:#1}}
\newcommand{\rco}[1]{Corollary~\ref{C:#1}}

\newcommand{\rt}[1] {Theorem~\ref{T:#1}}

\newcommand{\fin}{{\operatorname{fin}}}

\newcommand{\pr}{\operatorname{pr}}

\newcommand{\Aut}{\operatorname{Aut}}

\newcommand{\Ad}{\operatorname{Ad}}

\newcommand{\red}{{\operatorname{red}}}

\newcommand{\ev}{\operatorname{ev}}

\newcommand{\Fl}{\mathcal{F}\ell}
\newcommand{\Spr}{\operatorname{Spr}}

\newcommand{\St}{\operatorname{St}}

\newcommand{\codim}{\operatorname{codim}}

\newcommand{\rss}{{\operatorname{rss}}} \newcommand{\tn}{{\operatorname{tn}}}

\renewcommand{\si}{\sigma}

\newcommand{\CN}{{\mathcal{N}}}

\newcommand{\ft}{\frak{t}}

\newcommand{\fc}{\frak{c}}

\newcommand{\br}{\mathbf{r}}
\newcommand{\fk}{\mathsf k}
\newcommand{\fP}{\frak{P}}

\begin{document}

\title[$S$-cells in affine Weyl groups]
{Lusztig conjectures on $S$-cells in affine Weyl groups}

\author{Michael Finkelberg}
\address{Department of Mathematics\\
National Research University Higher School of Economics\\
Russian Federation, Usacheva st.\ 6, 119048, Moscow\\
\newline Skolkovo Institue of Science and Technology\\
\newline Institute for Information Transmission Problems of RAS}
\email{fnklberg@gmail.com}
\date{\today}

\author{David Kazhdan}
\address{Institute of Mathematics\\
The Hebrew University of Jerusalem\\
Givat-Ram, Jerusalem,  91904\\
Israel} \email{kazhdan@math.huji.ac.il}

\author{Yakov Varshavsky}
\address{Institute of Mathematics\\
The Hebrew University of Jerusalem\\
Givat-Ram, Jerusalem,  91904\\
Israel} \email{vyakov@math.huji.ac.il}

\dedicatory{To George Lusztig on his 75th birthday with admiration}

\begin{abstract}
We apply the dimension theory developed in~\cite{BKV} to establish some of Lusztig's conjectures~\cite{Lu} on $S$-cells in affine Weyl groups.
\end{abstract}
\maketitle

\tableofcontents

\section*{Introduction}

\begin{Emp} \label{E:scells}
  {\bf $S$-cells.}
  Following~\cite[0.1]{Lu}, recall that for a connected complex reductive group $G$,
  its Weyl group $W_\fin$ is partitioned into
  $S$-cells:\footnote{$S$ stands for Steinberg, Spaltenstein and Springer.}
  $W_\fin=\bigsqcup_{{\mathbb O}\in{\mathfrak U}}W_{\mathbb O}$ parameterized by the set $\mathfrak U$
  of nilpotent $G$-orbits in ${\mathfrak g}=\operatorname{Lie}G$ as follows.
  Given $w\in W_\fin$, we take Borel subalgebras ${\mathfrak b},{\mathfrak b'}\subset{\mathfrak g}$
  in relative position $w$ and consider the intersection
  ${\mathfrak n}_{\mathfrak b}\cap{\mathfrak n}_{{\mathfrak b}'}$ of their nilpotent radicals.
  There is a unique nilpotent orbit ${\mathbb O}$ such that the intersection
  ${\mathbb O}\cap{\mathfrak n}_{\mathfrak b}\cap{\mathfrak n}_{{\mathfrak b}'}$ is open in
  ${\mathfrak n}_{\mathfrak b}\cap{\mathfrak n}_{{\mathfrak b}'}$. By definition, $w\in W_{\mathbb O}$.

  For any nilpotent orbit $\mathbb O$, the $S$-cell $W_{\mathbb O}$ is the image of a map
  $\varpi\colon [\operatorname{Spr}_a]\times[\operatorname{Spr}_a]\to W_\fin$ defined as follows:
 let $a\in {\mathbb O}$ be an arbitrary element, let $\operatorname{Spr}_a$ be the Springer fiber over $a$, that is, the space of Borel subalgebras
 containing $a$, let $[\operatorname{Spr}_a]$ be the set of the irreducible components of
 $\operatorname{Spr}_a$, and finally
  $\varpi(X,X')\in W_{\fin}$ is the relative position of generic points of
  irreducible components $X,X'\in[\Spr_a]$.

  Indeed, recall the
  Springer resolution $\mu\colon T^*{\mathcal B}=\widetilde{\mathcal U}\to{\mathcal U}$,
  where ${\mathcal B}$ is the flag variety of $G$, and ${\mathcal U}\subset{\mathfrak g}$ is
  the nilpotent cone. It is known that $\mu$ is strictly semismall, i.e.\ for any nilpotent orbit
  ${\mathbb O}\subset{\mathcal U}$, its codimension in $\mathcal U$ is exactly twice the dimension
  of the Springer fiber $\operatorname{Spr}_a=\mu^{-1}(a)$ for any $a\in{\mathbb O}$. In other
  words, all the nilpotent orbits are the relevant strata~\cite[1.1]{BM} of the Springer morphism $\mu$.
  The strict semi-smallness of $\mu$ implies that the Steinberg variety of triples
  $\St_G:=\widetilde{\mathcal U}\times_{\mathcal U}\widetilde{\mathcal U}$ is equidimensional
  of dimension $2\dim{\mathcal B}$. On the other hand, the irreducible components of $\St_G$ are nothing
  but the conormal bundles $T^*_{{\mathbb O}_w}({\mathcal B}\times{\mathcal B})$ to orbits of $G$ acting diagonally on
  ${\mathcal B}\times{\mathcal B}$ (such orbits are pairs of Borel subalgebras in relative
  position $w\in W_\fin$). Thus both $W_{\mathbb O}$ and
  $\varpi([\operatorname{Spr}_a]\times[\operatorname{Spr}_a])$ parameterize the set of irreducible
  components of $\St_G$ whose generic points lie above the generic point of ${\mathbb O}$.
\end{Emp}

\begin{Emp}
  {\bf Affine $S$-cells.}
  In case $G$ is almost simple simply connected, Lusztig~\cite{Lu} suggested two definitions of
  a partition of the affine Weyl group $W=\bigsqcup_{w\in W_\fin/\!\Ad}W_w$ into affine
  $S$-cells parameterized by the conjugacy classes of $W_\fin$, and conjectured the equivalence of
  the two definitions.

The goal of this work is to prove a weak form of Lusztig's conjecture replacing the argument
of~\re{scells} by its affine analog. In the affine case, the role of the nilpotent cone ${\mathcal U}$ is played by the
  space of topologically nilpotent elements  ${\mathcal N}\subset L{\mathfrak g}={\mathfrak g}(\!(t)\!)$ in the loop Lie algebra of $\mathfrak g$, while the role of the partition
  ${\mathcal U}=\bigsqcup_{{\mathbb O}\in{\mathfrak U}}{\mathbb O}$ is played
  by the Goresky--Kottwitz--MacPherson stratification~\cite{GKM} of ${\mathcal N}$.
  The affine Springer resolution $\widetilde{\mathcal N}\to{\mathcal N}$ is semi-small, but
  not strictly semi-small; the relevant strata are parameterized by $W_\fin/\!\Ad$~\cite[Lemma 4.4.4(d)]{BKV}
  (in particular,~\cite[Conjecture 3.3]{Lu} follows).
  This implies a weak form of Lusztig's conjecture~\cite[2.3]{Lu}:
  the equivalence of the two definitions not for
  arbitrary elements of the relevant GKM strata, but only for generic elements.
  As a consequence, we show Lusztig's conjecture~\cite[2.4]{Lu} asserting that for any
  $w\in W_\fin/\!\Ad$, the corresponding $S$-cell $W_w\subset W$ is non-empty, and
  that $W_w$ is finite if and only if $w$ is elliptic.

  Note that all geometric objects involved are infinite dimensional ind-schemes, therefore the classical notion of dimension does not make sense
  in this setting. Instead we apply the dimension theory developed in~\cite{BKV}.


In case $G$ is of type $A$ it is expected that affine $S$-cells coincide with the two-sided
Kazhdan–Lusztig cells (the latter cells are explicitly described in~\cite{L}), see in
particular~\cite[1.4]{Lu} where this is pointed out for
$G=\operatorname{SL}(3)$ and~\cite{La} where this is established for related (but a priori
different) $\tilde S$-cells defined in~\cite[Section 4]{Lu}.
In certain (rectangular) special cases this follows from the recent result of~\cite{BYY}
together with our main theorem. More precisely, in~\cite{BYY} relative positions of generic
points in components of certain affine Springer fibers are computed, the answer turns out
to be related to Kazhdan-Lusztig cells (as conjectured by R.~Bezrukavnikov to hold more
generally for groups of type A). These relative positions are related to $S$-cells by our
main theorem.
\end{Emp}


  In the next six subsections we provide definitions and more
precise formulations of the results.


\begin{Emp} \label{E:affstvar}
{\bf The affine Steinberg variety.}
(a) Let $G$ be a connected reductive group over an algebraically closed field ${\mathsf k}$, $W_\fin$ the Weyl group of $G$, and $R$ the set of roots of $G$. We assume that the characteristic of ${\mathsf k}$ does not divide the order of $W_\fin$.

(b) Let $\C{L}G$ be the loop group of $G$, $I\subset \C{L}G$ an Iwahori subgroup scheme, and $\Fl=\C{L}G/I$ the affine flag variety.
We denote by $\g$ the Lie algebra of $G$, by $\C{L}\g$ the corresponding loop algebra, and by $\C{I}^+\subset \C{L}\g$ the Lie algebra of the prounipotent radical $I^+$ of $I$. More generally, for every $[g]\in \Fl$, we set $\C{I}^+_g:=\Ad_g(\C{I}^+)$.

(c) Let $\CN\subset \C{L}\g$ be the locus of topologically nilpotent elements of $\C{L}\g$.
More precisely, let $\fc$ be the Chevalley space of $\g$, $\C{L}^+(\fc)\subset\C{L}\fc$ be the
arc and the loop spaces of $\fc$, respectively,
$\ev\colon \C{L}^+(\fc)\to\fc$ the evaluation map, and $\C{L}\chi\colon \C{L}\g\to\C{L}\fc$ the morphism,
induced by the canonical morphism $\chi\colon \g\to\fc$. Then we denote by $\C{L}^+(\fc)_\tn:=\ev^{-1}(0)\subset \C{L}^+(\fc)$ the locus of
topologically nilpotent elements, and set $\CN:=\C{L}\chi^{-1}(\C{L}^+(\fc)_\tn)\subset \C{L}G$.

(d) Let  $\wt{\CN}$ be the affine Springer resolution of $\CN$, which is a closed
ind-subscheme of $\CN\times \Fl$ consisting of points $(\gm,[g])$ such that $\gm\in\C{I}^+_g$.

(e) The affine Steinberg variety is the fibered product $\St:=\wt{\CN}\times_\CN \wt{\CN}$. It is a closed ind-subscheme of $\CN\times \Fl\times \Fl$ consisting of points $(\gm,[g'],[g''])$ such that $\gm\in\C{I}^+_{g',g''}:=\C{I}^+_{g'}\cap \C{I}^+_{g''}$.
\end{Emp}

\begin{Emp} \label{E:str by orbits}
{\bf Stratification by $\C{L}G$-orbits.}
(a) Let $f:X\to Y$ be a morphism of ind-schemes (or stacks). Then every stratification $\{Y_{\al}\}_{\al\in A}$ of $Y$ by locally closed sub-ind-schemes (or stacks) over $\fk$ induces a stratification $\{X_{\al}\}_{\al\in A}$ of $X$ such that $X_{\al}=f^{-1}(Y_{\al})$ for all $\al\in A$.

(b) Recall that there is a natural bijection $x\mapsto (\Fl\times \Fl)^x:=\C{L}G(1,x)$ between elements of the extended affine Weyl group $W$ of $G$ and $\C{L}G$-orbits in $\Fl\times \Fl$. In particular, we get a stratification $\{(\Fl\times \Fl)^x\}_{x\in W}$ of $\Fl\times \Fl$.

(c) Combining (a) and (b), we get a stratification $\{\St^x\}_{x\in W}$ of $\St$.
\end{Emp}

\begin{Emp} \label{E:gkm}
{\bf The Goresky--Kotwitz--MacPherson stratification.}
(a) As in~\cite{GKM} and~\cite[3.3.4]{BKV} the regular semisimple part $\C{L}^+(\fc)^\rss:=\C{L}^+(\fc)\cap(\C{L}\g)^{\rss}$ of $\C{L}^+(\fc)$ has a natural stratification by finitely presented locally closed irreducible subschemes $\fc_{w,\br}$, parameterized by $W_\fin$-orbits of pairs $(w,\br)$, where

\quad\quad $\bullet$ $w$ is an element of the Weyl group $W_\fin$,

\quad\quad $\bullet$ $\br$ is a function $R\to\B{Q}_{\geq 0}$, and

\quad\quad $\bullet$ $W_{\fin}$ acts by the formula $u(w,\br)=(uwu^{-1},u(\br))$ for all $u\in W_{\fin}$.

(b) Namely, denote by $h$ the order of $W_{\fin}$, fix the primitive $h$-th root of unity $\xi\in\fk$, and let
$\si\in\Aut(\fk[\![t^{1/h}]\!]/\fk[\![t]\!])$ be the automorphism given by the formula $\si(t^{1/h})=\xi t^{1/h}$.
Let $\ft$ be the abstract Cartan Lie algebra of $\g$, and let $\ft\to\fc$ be the natural projection. Then every $z\in \C{L}^+(\fc)(\fk)=\fc(\fk[\![t]\!])$ has a lift $\wt{z}\in \ft(\fk[\![t^{1/h}]\!])$.

The GKM stratification of $\C{L}^+(\fc)^\rss$ is characterized by the condition that $z\in \C{L}^+(\fc)_{w,\br}$ if and only if we have $\si(\wt{z})=w^{-1}(\wt{z})$ and  $\br(\al)$ equals the valuation of $\al(\wt{z})\in \fk(\!(t^{1/h})\!)^{\times}$ for all $\al\in R$.

(c) Applying observation of~\re{str by orbits}(a) to the projection $\C{L}\chi\colon \CN\to \C{L}^+(\fc)$, we get that the GKM stratification of
$\C{L}^+(\fc)^{\rss}$ induces stratifications of the regular semisimple part of $\CN$ and hence also of $\C{I}^+,\wt{\CN}$ and $\St$.
Note that if the stratum $\CN_{w,\br}$ (resp. $\C{I}^+_{w,\br}$) is non-empty, then $\br>0$, that is, $\br(\al)>0$ for every $\al\in R$.

(d) For every $w\in W_\fin$, we denote by $\ft_w$ the twisted form of $\ft$ over $\C{O}$ (see~\cite{GKM} or~\cite[3.3.3]{BKV}). The GKM stratification $\fc_{w,\br}$ of $\C{L}^+(\fc)^{\rss}$ induces a stratification
$\ft_{w,\br}$ of $\C{L}^+(\ft_w)^{\rss}$. Let $\C{L}^+(\ft_w)_\tn\subset \C{L}^+(\ft_w)$ be the locus of topologically nilpotent elements. Then we have an inclusion $\ft_{w,\br}\subset \C{L}^+(\ft_w)_\tn$ if and only if $\br>0$.
\end{Emp}

\begin{Emp} \label{E:min}
{\bf Minimal GKM pairs.} (a) We call a pair $(w,\br)$, where $w\in W_\fin$ and $\br$ is a function $R\to\B{Q}_{>0}$, is a {\em GKM pair},
if the stratum  $\fc_{w,\br}$ of  $\C{L}^+(\fc)^\rss$ is non-empty. We denote the set of $W_{\fin}$-orbits of GKM pairs by $\frak{P}$, and for
every GKM pair $(w,\br)$ we denote its class in $\fP$ by $[w,\br]$.

(b) Fix $w\in W_{\fin}$. We call a GKM pair $(w,\br)$ (or its class $[w,\br]\in\fP$) {\em minimal}, if the stratum $\ft_{w,\br}\subset\C{L}^+(\ft_w)_\tn$ is open. Explicitly, $\br$ is a minimal element among functions $R\to\B{Q}_{>0}$
such that $(w,\br)$ is a GKM pair.

(c) Notice that since each $\C{L}^+(\ft_w)_\tn$ is irreducible (see~\cite[3.4.4]{BKV}), it has a unique open GKM stratum
$\ft_{w,\br}$. Namely, it is the GKM stratum, containing the generic point of $\C{L}^+(\ft_w)_\tn$. Therefore for each $w\in W_{\fin}$ there exists a unique minimal GKM pair $(w,\br)$.

(d) We denote by $\frak{P}_{\min}\subset\frak{P}$ the set of minimal classes in $\frak{P}$.
\end{Emp}

Following Lusztig, we are now going to relate the two stratifications of the affine Steinberg variety $\St$ defined above.

\begin{Emp} \label{E:main}
{\bf Main construction.} Recall that $\St$ is a closed ind-subscheme of the product $\CN\times (\Fl\times \Fl)$.

(a) By~\re{str by orbits}(a) and~\re{gkm}(c), for every pair $(g',g'')\in \Fl\times \Fl$, the regular semisimple part of the fiber $\St^{g',g''}\subset \CN$ is equipped with a GKM-stratification $\{\St^{g',g''}_{w,\br}\}_{[w,\br]\in\frak{P}}$.
Similarly, for every $\gm\in \CN$, the reduced Steinberg fiber $\St_{\gm}\subset \Fl\times \Fl$ is equipped with a stratification $\{\St_{\gm}^x\}_{x\in W}$ (see~\re{str by orbits}(a),(b)).

(b)  Since $\St^{g',g''}=\C{I}^+_{g',g''}$ is irreducible while every GKM stratum $\St^{g',g''}_{w,\br}\subset\St^{g',g''}$ is a finitely presented locally closed subscheme, there exists a unique class $\wt{\pi}(g',g')=[w,\br]\in\fP$ such that the stratum
$\St^{g',g''}_{w,\br}\subset \St^{g',g''}$ is open (compare~\re{min}(c)). Moreover, since the GKM stratification of $\CN$ is $\C{L}G$-equivariant, the class  $\wt{\pi}(g',g')$ only depends on the $\C{L}G$-orbit of $(g',g'')$.

(c) By (b) and~\re{str by orbits}(b), for every $x\in W$ there exists a unique class $\pi(x)=[w,\br]\in\fP$ such that $\wt{\pi}(g',g'')=[w,\br]$ for every $(g',g'')\in (\Fl\times \Fl)^x$. We also denote by $\ov{\pi}(x):=[w]\in W_{\fin}/\!\Ad$ the conjugacy class of $w$.

(d) Assume from now on that  $\gm\in \CN\subset\C{L}\g$ is regular semisimple. Then the reduced affine Springer fiber $\Fl_{\gm}$ is an  equidimensional scheme locally of finite type over $\mathsf k$ (see~\cite{KL}). Hence the same is true for $\St_{\gm}=\Fl_{\gm}\times\Fl_{\gm}$. Moreover, by the formula of Bezrukavnikov--Kazhdan--Lusztig~\cite{Be}, for every class $[w,\br]\in\fP$ there exists $\dt_{w,\br}\in\B{Z}_{\geq 0}$ such that $\dim\Fl_{\gm}=\dt_{w,\br}$ for every $\gm\in \CN_{w,\br}$.

(e) Following Lusztig, we define a subset $\si(\gm)\subset W$ to be the set of all $x\in W$ such that the locally closed subscheme
$\St^x_{\gm}\subset \St_{\gm}$ is of full dimension  $\dim \St_{\gm}=2\dt_{w,\br}$. Alternatively, $x\in \si(\gm)$ if and only if there exist
irreducible components $C',C''$ of $\Fl_{\gm}$ such that $(C'\times C'')^x\subset C'\times C''$ is an open subscheme.
\end{Emp}

\begin{Emp} \label{E:conj}
  {\bf Lusztig's conjectures.} Lusztig conjectured that  the two maps defined above are closely
  connected. More precisely, Lusztig~\cite[Conjectures~3.3 and~2.3]{Lu} conjectured that

(a) For every $x\in W$, the class $\pi(x)=[w,\br]\in\fP$ is minimal.

(b) For every $[w,\br]\in\fP_{\min}$ and $\gm\in \CN_{w,\br}$, we have an equality $\si(\gm)=\pi^{-1}([w,\br])$.
In other words, for $[w,\br]\in\fP_{\min}$, $x\in W$ and $\gm\in \CN_{w,\br}$, we have $\pi(x)=[w,\br]$ if and only if
$\dim \St^x_{\gm}=2\dt_{w,\br}$.

\vskip 4truept

\noindent Lusztig~\cite[~2.4]{Lu} also remarked that assertions (a) and (b) imply that

\vskip 4truept

(c) For every $w\in W_{\fin}$, the preimage $\ov{\pi}^{-1}([w])$ is non-empty.

(d) Assume that $G$ is semisimple. Then $\ov{\pi}^{-1}([w])$ is finite if and only if $w$ is elliptic.
\end{Emp}

\begin{Emp}
  {\bf What is done in this work?} Our goal is to prove~Conjecture~\re{conj}(a) and to show
  that~Conjecture~\re{conj}(b) holds for ``generic" elements. More precisely, we show the
  existence of an $\C{L}G$-invariant open dense sub-indscheme ${}^x\CN_{w,\br}\subset \CN_{w,\br}$
  (depending on $x\in W$) such that for every $\gm\in {}^x\CN_{w,\br}$, we have $\pi(x)=[w,\br]$
  if and only if $\dim \St^x_{\gm}=2\dt_{w,\br}$. As a consequence, we deduce
  Conjectures~\re{conj}(c),(d). Finally, we show that the full Conjecture~\re{conj}(b)
  follows from a certain flatness conjecture.
\end{Emp}

\begin{Emp} \label{E:strategy}
{\bf Our strategy.} (a) To every morphism $f\colon X\to Y$ of schemes of finite type over $\fk$ we associate a dimension function
$\un{\dim}_f\colon X\to\B{Z}$ given by $\un{\dim}_f(z):=\dim_z X-\dim_{f(z)}Y$ for $z\in X$. 

(b) Our dimension function satisfies the property that for every $z\in X$ we have an inequality $\un{\dim}_f(z)\leq\dim_z f^{-1}(f(z))$ and that there exists an open dense subset $U\subset Y$ such that we have an equality $\un{\dim}_f(z)=\dim_z f^{-1}(f(z))$  for every $z\in f^{-1}(U)$.

(c) Our main observation is that the dimension function of (a) can be defined for locally
finitely presented morphisms between certain infinite-dimensional schemes, and that
property~(b) still holds in this case. Namely, it can be done when $Y$ is {\em placid}, that is,
locally has a presentation as a limit $Y\simeq\lim_i Y_i$, where each $Y_i$ is of finite type,
and all transition maps are smooth affine.

(d) Fix $x\in W$ and $[w,\br]\in\fP$. We would like to apply the construction~(c)
to the projection $p\colon\St^x_{w,\br}\to\CN_{w,\br}$.
Unfortunately, we can not do it directly, because both source and target of $p$ are infinite dimensional ind-schemes,
rather than schemes of finite type. To overcome this, we observe that the projection $p$ is $\C{L}G$-equivariant,
and there exists a natural embedding $\ft_{w,\br}\hra\CN_{w,\br}$, unique up to an $\C{L}G$-conjugacy
such that the composition $\ft_{w,\br}\hra\CN_{w,\br}\to[\C{L}G\bs \CN_{w,\br}]$ is surjective.
Therefore we can replace $p$ by its pullback $p_{\ft}\colon\St^x_{\ft,w,\br}\to\ft_{w,\br}$ to
$\ft_{w,\br}\subset\CN_{w,\br}$.

It turns out that the reduced ind-scheme  $(\St^x_{\ft,w,\br})_{\red}$ is actually a scheme,
locally finitely presented over $\ft_{w,\br}$, therefore
the construction of~(c) applies to $p_{\ft,\red}\colon(\St^x_{\ft,w,\br})_{\red}\to\ft_{w,\br}$. Furthermore, there is a discrete group $\La'$ acting freely and
discretely on $\St^x_{\ft,w,\br}$ over $\ft_{w,\br}$ such that the quotient
$[\La'\bs (\St^x_{\ft,w,\br})_{\red}]$ is a scheme, finitely presented over $\ft_{w,\br}$.
Thus an analog of~(b) applies to $p_{\ft,\red}$ as well.

(e) Our main technical result asserts that function $\un{\dim}_{p_{\ft},\red}$ equals
$2\dt_{w,\br}+a^+_{w,\br}-\un{b}(x)_{w,\br}^+$, where $a^+_{w,\br}$ is a non-negative integer such that
$a^+_{w,\br}=0$ if and only if the class $[w,\br]\in\fP$ is minimal, and $\un{b}(x)_{w,\br}^+$
is a non-negative function such that $\un{b}(x)_{w,\br}^+=0$ if and only if $\pi(x)=[w,\br]$.

(f) Both Conjecture~\re{conj}(a) and a weak form of Conjecture~\re{conj}(b) easily follow from the combination of~(e)
and~(b). Namely, when $\pi(x)=[w,\br]$, these assertions imply that for a generic
$\gm\in \ft_{w,\br}$, we have an inequality $\dim\St^x_{\gm}\geq 2\dt_{w,\br}+a^+_{w,\br}$, which implies
that $a^+_{w,\br}=0$, thus $[w,\br]$ is minimal. Conversely, if $[w,\br]$ is minimal, then for a
generic $\gm\in \ft_{w,\br}$, we have an equality
$\dim_{\wt{\gm}}\St^x_{\gm}=2\dt_{w,\br}-\un{b}(x)^+_{w,\br}(\wt{\gm})$ for every
$\wt{\gm}\in \St^x_{\gm}$, which implies that $\dim\St^x_{\gm}=2\dt_{w,\br}$ if and only if
$\pi(x)=[w,\br]$.
\end{Emp}

\begin{Emp}
  {\bf Plan of the paper.} The paper is organized as follows. In the first two sections we
  introduce our main ingredients, namely placid stacks
  and dimension functions, mostly repeating the corresponding parts from~\cite{BKV}. Then, in the
  next three sections we prove Lusztig conjecture~\re{conj}(a) and a weak form of~\re{conj}(b), and
  deduce conjectures~\re{conj}(c),(d) from them. Finally, in last section we deduce the full
  Lusztig conjecture~\re{conj}(b) from a certain flatness conjecture.
\end{Emp}

\begin{Emp}
  {\bf Acknowledgments.} This paper is based on the work~\cite{BKV} of two of us joint with
  Alexis Bouthier. We also thank Alexis Bouthier for his comments on the first draft of this work.

  Similar results were independently obtained by Zhiwei Yun (private communication). We are grateful to him for an interesting discussion.
  We are also grateful to Roman Bezrukavnikov for explanations about~\cite{BYY}.

  The research of Y.V.\ was partially supported by the ISF grant 822/17.
  The project have received funding from ERC under grant agreement No 669655.
  M.F.\ was partially funded within the framework of the HSE University Basic Research Program
  and the Russian Academic Excellence Project `5-100'.
\end{Emp}

\section{Placid stacks}

In this and the next sections we will review the material that appears in~\cite{BKV}.
To make the exposition simpler, most of our notions are more restrictive than those considered in~\cite{BKV}.

\begin{Emp} \label{E:glplsch}
{\bf Schemes admitting placid presentations.} (a) We say that a scheme $X$ over ${\mathsf k}$ {\em admits a placid presentation},
if it has a presentation
$X\simeq\lim_{i\in\B{N}} X_i$, where each $X_i$ is a scheme of finite type over ${\mathsf k}$, and every projection $X_{i+1}\to X_i$ is smooth and affine. 

(b) Let $f\colon Y\to X$ be a finitely presented morphism of schemes such that $X$ admits a placid presentation $X\simeq\lim_i X_i$. Then there exists an index $i$ and a morphism $f_i\colon Y_i\to X_i$ of schemes of finite type over ${\mathsf k}$ such that $f$ is a pullback of $f_i$. In particular, $Y\simeq\lim_{j\geq i}(Y_i\times_{X_i} X_j)$ is a placid presentation of $Y$.

(c) We call a morphism of schemes $f\colon X\to Y$ {\em strongly pro-smooth}, if $X$ has a presentation $X\simeq\lim_i X_i$ over $Y$, where $X_0\to Y$ is smooth and finitely presented, while all projections $X_{i+1}\to X_i$ are smooth, finitely presented and affine.

(d) The class of (c) is closed under compositions and pullbacks (see~\cite[1.1.3]{BKV}). It follows that if $f\colon X\to Y$ strongly pro-smooth, and $Y$ admits a placid presentation, then $X$ admits a placid presentation as well.

(e) Notice that a scheme $X$ admitting a placid presentation is irreducible if and only if has a placid presentation $X\simeq\lim_i X_i$ such that
 $X_i$ is irreducible for all $i$.
\end{Emp}


\begin{Emp} \label{E:plalgsp}
{\bf Placid algebraic spaces and smooth morphisms.}
(a) We call a scheme/an algebraic space $X$ {\em placid}, if it has an \'etale covering by schemes admitting placid presentations. Using~\re{glplsch}(b), one deduces that if $f\colon X\to Y$ is a locally finitely presented morphism of algebraic spaces and $Y$ is placid, then $X$ is placid.

(b) We call a morphism $f\colon X\to Y$ of algebraic spaces {\em smooth}, if locally in
the \'etale topology it is a strongly pro-smooth morphism of schemes. Explicitly this means that
there exist \'etale coverings $\{Y_{\al}\}_{\al}$  of $Y$ and $\{X_{\al,\beta}\}_{\beta}$ of
$f^{-1}(Y_{\al})=X\times_Y Y_{\al}$ by schemes such then every
$X_{\al,\beta}\to Y_{\al}$ is strongly pro-smooth. Using~\re{glplsch}(d) one sees that if
$f\colon X\to Y$ is a smooth morphism of algebraic spaces and $Y$ is placid, then $X$ is placid as well.

(c) The class of smooth morphisms is closed under compositions and pullbacks (by~\re{glplsch}(d)).

(d) As in~\cite{BKV}, our smooth morphisms are not assumed to be locally finitely presented. On the other hand,
all smooth morphisms are automatically flat.
\end{Emp}

\begin{Emp}
{\bf Remark.}
For the purpose of this work, we could avoid talking about algebraic spaces, and restrict ourselves to schemes instead
(compare~\re{remschemes}). Furthermore, all placid schemes appearing in this work have Zariski open coverings by schemes
admitting placid presentations.
\end{Emp}

\begin{Emp} \label{E:plstacks}
{\bf Placid stacks.}
(a) By a {\em stack} over $\fk$, we mean a stack in groupoids in the \'etale topology. Using observation~\re{plalgsp}(c),
we can talk about smooth representable morphisms between stacks.

(b) A stack $\C{X}$ over $\fk$ is called {\em placid}, if there exists a smooth representable surjective morphism
$X\to\C{X}$ from a placid algebraic space $X$. Such a map is called a {\em placid atlas}.

(c) A representable morphism of stacks $f\colon \C{X}\to\C{Y}$ is called {\em (locally) finitely presented}, if for every morphism $Y\to\C{Y}$ from an algebraic space  $Y$, the pullback $\C{X}\times_{\C{Y}}Y\to Y$ is a (locally) finitely presented morphism of algebraic spaces.

(d) Assume that in the situation of (c) the stack $\C{Y}$ is placid. Then $\C{X}$ is placid
as well. Indeed, if $Y\to \C{Y}$ is a placid atlas, then   $\C{X}\times_{\C{Y}}Y\to \C{X}$ is
a placid atlas by~\re{plalgsp}(c).
\end{Emp}

\begin{Emp} \label{E:basic example}
{\bf Example.} Let $G$ be a strongly pro-smooth group scheme acting on a placid algebraic space $X$. Then the quotient stack
$\C{X}=[G\bs X]$ is placid, and the projection $X\to\C{X}$ is a placid atlas.
\end{Emp}

\begin{Emp}
{\bf The underlying set.}
(a) Recall that to every stack $\C{X}$ over $\fk$, one associates the underlying set $\un{\C{X}}$, whose points are equivalent classes
of pairs $(K,z)\in \C{X}(K)$, where $K$ is a field extension of $\fk$, $z\in \C{X}(K)$ and $(z_1,K_1)\sim (z_2,K_2)$, if there exists a larger field
$K\supset K_1,K_2$ such that points $z_1|_K,z_2|_K\in\C{X}(K)$ are isomorphic.

(b) Note that when $X$ is an algebraic space, then $\un{X}$ is the underlying set of $X$. More generally,
if $\C{X}$ is the quotient stack $[G\bs X]$ as in~\re{basic example}, the $\un{\C{X}}$ is the set of orbits $\un{G}\bs\un{X}$.

(c) To simplify the notation, we will denote the set $\un{\C{X}}$ simply by $\C{X}$.
\end{Emp}

\begin{Emp} \label{E:red}
{\bf Reduction.}
(a) Recall that to every scheme/algebraic space $X$ one can associate the corresponding reduced scheme/algebraic space $X_{\red}$. Moreover, $X_{\red}$ is placid, if $X$ is such (see~\cite[Lemma 1.4.5]{BKV}).

(b) More generally, to every placid stack $\C{X}$ one can associate a reduced  placid stack $\C{X}_{\red}$ (see~\cite[1.4]{BKV}).
Furthermore, the assignment $\C{X}\mapsto\C{X}_{\red}$ is functorial, we have a canonical functorial finitely presented closed embedding  $\C{X}_{\red}\to\C{X}$, and the induced map $\un{\C{X}_{\red}}\to\un{\C{X}}$ of the underlying sets is a bijection.
\end{Emp}

\section{Dimension theory}

\begin{Emp} \label{E:dimft}
{\bf Dimension function: schemes of finite type.}
(a) To every map of sets $f\colon X\to Y$ and a function $\phi\colon Y\to\B{Z}$, we associate the function  $f^*(\phi)=\phi|_X:=\phi\circ f\colon X\to\B{Z}$.

(b) For a scheme $X$ of finite type over $\fk$ and $z\in X$, we denote by $\dim_z(X)$ is the maximum of dimensions of irreducible components of $X$, containing $z$. As in~\cite[2.1.1]{BKV}, one associates to $X$ a dimension function $\un{\dim}_X\colon X\to\B{Z}$, defined by $\un{\dim}_X(z)=\dim_z(X)$ for every $z\in X$.

(c) Then, as in~\cite[2.1.2]{BKV}, to every morphism $f\colon X\to Y$ between schemes of finite type over $\fk$, we associate the dimension function
\[
\un{\dim}_f=\un{\dim}(X/Y):=\un{\dim}_X-f^*(\un{\dim}_Y)\colon X\to \B{Z}.
\]
In other words, we define  $\un{\dim}_f(z):=\dim_z(X)-\dim_{f(z)}(Y)$ for every $z\in X$.
%
%
\end{Emp}

Next we are going to extend these notions to placid schemes and stacks.

\begin{Emp} \label{E:dimplst}
{\bf Dimension function: placid stacks} (see~\cite[Lemmas 2.2.4 and 2.2.5]{BKV}).

(a) For every finitely presented morphism $f\colon X\to Y$ of schemes admitting placid presentations, there exists a unique dimension function
$\un{\dim}_f=\un{\dim}(X/Y)\colon X\to\B{Z}$ such that for every placid presentation $Y\simeq\lim_i Y_i$ of $Y$ and morphism $f_i\colon X_i\to Y_i$ as in~\re{glplsch}(b), we have $\un{\dim}_f=\pi_i^*(\un{\dim}_{f_i})$, where $\un{\dim}_{f_i}$ was defined in \re{dimft}, and $\pi_i\colon X\to X_i$ is the projection. In other words, we have $\un{\dim}_f(z)=\un{\dim}_{f_i}(\pi_i(z))$ for every $z\in X$.

(b) For every locally finitely presented morphism $f\colon X\to Y$ of placid algebraic spaces,
there exists a unique function
$\un{\dim}_f=\un{\dim}(X/Y)\colon X\to\B{Z}$ such that for every commutative diagram
\begin{equation} \label{Eq:cddimalgsp}
\begin{CD}
X' @>f'>> Y'\\
@V h VV @VV g V\\
X @>f>> Y,
\end{CD}
\end{equation}
where $f'\colon X'\to Y'$ is a finitely presented morphism of schemes admitting placid presentations, and $g$ and $h$ are \'etale, we have an equality $h^*(\un{\dim}_f)=\un{\dim}_{f'}$, where
$\un{\dim}_{f'}$ was defined in (a).

(c) For every representable locally finitely presented morphism $f\colon \C{X}\to \C{Y}$ of placid stacks, there exists a unique function
$\un{\dim}_f=\un{\dim}(\C{X}/\C{Y})\colon \un{\C{X}}\to\B{Z}$ such that for every Cartesian diagram
\[
\begin{CD}
X @>f'>> Y\\
@V h VV @VV g V\\
\C{X} @>f>> \C{Y},
\end{CD}
\]
where $g$ and $h$ are placid atlases, we have
$h^*(\un{\dim}_f)=\un{\dim}_{f'}$, where
$\un{\dim}_{f'}$ was defined in (b).
\end{Emp}

\begin{Emp} \label{E:bex2}
{\bf Example.} In the situation of~\re{basic example}, let $f\colon Y\to X$ be a $G$-equivariant finitely presented morphism of
algebraic spaces. Then $f$ induces a finitely presented morphism $[f]\colon  [G\bs Y]\to [G\bs X]$ between quotient stacks, and our construction~\re{dimplst}(c) says that the function $\un{\dim}_f\colon Y\to \B{Z}$ is the pullback of
$\un{\dim}_{[f]}\colon G\bs Y\to \B{Z}$.
\end{Emp}

\begin{Emp} \label{E:propdim}
{\bf Properties.}
(a) The dimension function is {\em additive}, that is, for every pair $\C{X}\overset{f}{\to} \C{Y}\overset{g}{\to} \C{Z}$ of morphisms
as in \re{dimplst}(c), we have an equality $\un{\dim}_{gf}=\un{\dim}_f+f^*(\un{\dim}_g)$ (see~\cite[Lemma 2.2.5]{BKV}).

(b) For every $f$ as in \re{dimplst}(c), the induced morphism $f_{\red}\colon \C{X}_{\red}\to \C{Y}_{\red}$
is a representable locally finitely presented morphism of placid stacks as well (see~\re{red}(b)), and the dimension function
$\un{\dim}_{f_{\red}}\colon \un{\C{X}_{\red}}\to\B{Z}$ is the pullback of $\un{\dim}_{f}$ (see~\cite[Corollary 2.2.8]{BKV}).
\end{Emp}

\begin{Emp}
{\bf Notation.}
(a) We say that a finitely presented representable $f$ is of {\em constant dimension}, if the dimension function $\un{\dim}_f$ is constant. In this case, we often write ${\dim}_f={\dim}(X/Y)$ instead of
$\un{\dim}_f=\un{\dim}(X/Y)$.

(b) For a finitely presented locally closed embedding $\iota\colon Y\hra X$, we define
$\un{\codim}_X(Y):=-\un{\dim}_{\iota}$. Again, we write ${\codim}_X(Y)$ instead of $\un{\codim}_X(Y)$, when $\iota$ is of constant dimension.
\end{Emp}

\begin{Lem} \label{L:dimfib}
Let $f\colon  X\to Y$ be a finitely presented morphism between placid algebraic spaces.

(a) For every $z\in X$, we have an inequality $\un{\dim}_f(z)\leq \dim_z (f^{-1}(f(z))$.

(b) If $f$ is open, then the inequality of~(a) is an equality for all $z\in X$.

(c) Set $\dim\emptyset=-\infty$. Then there exists an open dense subspace $U\subset Y$ such that the function
$y\mapsto \dim f^{-1}(y)$ is locally constant on $U$,
and for every $z\in f^{-1}(U)$, the inequality of~(a) is an equality.

(d) Assume that $X$ is non-empty, $Y$ is irreducible, and the inequality of~(a) is
an equality for all $z\in X$. Then for every $U$ as in (c) and every $y\in U$, the fiber $f^{-1}(y)$ is non-empty.
\end{Lem}

\begin{proof}
Assume first that  $X$ and $Y$ are schemes of finite type over $\fk$. In this case the assertions (a) and (b) are well-known (see, for example,~\cite[14.2.1]{EGA} or~\cite[0B2L]{St}).

Next, (c) is easy. Namely, shrinking $Y$, one can assume that every connected component of $Y$ is irreducible, thus
reduce to the case, when $Y$ is irreducible. Next, it is enough to show the assertion for the restriction $f_{\al}\colon X_{\al}\to Y$ of $f$ to each irreducible component of $X$, thus we can assume that $X$ is irreducible as well.  In this case, the assertion is standard.

Finally, to show (d) we let $Y'\subset Y$ be  the closure of $f(X)$. Then, by our assumption and (a), for every $z\in X$ we have
\[
\dim_z(X)-\dim_{f(z)}(Y')\leq \dim_z (f^{-1}(f(z))=\dim_z(X)-\dim_{f(z)}(Y),
\]
thus $\dim_{f(z)}(Y')=\dim_{f(z)}(Y)$. Since $X$ is non-empty and $Y$ is irreducible, this implies that $f$ is dominant, which implies the assertion.

Assume now that $X$ and $Y$ are schemes admitting placid presentations. Then $f$ is a pullback of a certain morphism $f'\colon Y'\to X'$ of schemes of finite type over $\fk$, and the assertion for $f$ follows from the corresponding assertion for $f'$. Namely, if $U'\subset X'$ satisfies the condition of the lemma for $f'$, then its preimage $U\subset X$  satisfies the condition for $f$.

The general case now easily follows. Indeed, choose an \'etale covering $\{Y_{\al}\}_{\al}$ of $Y$ by schemes $Y_{\al}$ admitting placid presentations. Then the assertion for $f$ follows from the corresponding assertion for $X\times_Y Y_{\al}\to Y_{\al}$. Thus we can assume
that $Y$ is a scheme admitting a placid presentation. Finally, choose an \'etale covering $X'\to X$ by a scheme admitting a placid presentation. Then the assertion for $f$ follows from the corresponding assertion for $X'\to X\overset{f}{\to}Y$.
\end{proof}




\section{Proof of Conjecture ~\re{conj}(a)}

We fix $x\in W$ and $[w,\br]\in\fP$.

\begin{Emp}
{\bf Notation.}
%
(a) Set $\C{Y}:=\wt{\CN}\times_\CN \C{I}^+$. Then $\C{Y}$ is a closed ind-subscheme of $\C{I}^+\times\Fl$.

(b) Using embedding $W\hra\Fl$, we can view $x$ as a point of $\Fl$, and set $\Fl^x:=Ix\subset\Fl$.
We denote by $\C{Y}^x\subset\C{Y}$ the preimage of $\Fl^x\subset\Fl$, and by
$\C{Y}^x_{w,\br}\subset\C{Y}^x$ the preimage of $\C{I}^+_{w,\br}\subset \C{I}^+$.

(c) Notice that $\C{I}^+$ is an affine scheme admitting a placid presentation, $\C{I}^+_{w,\br}\subset\C{I}^+$ is a finitely presented locally closed subscheme, while both projections $\C{Y}^x\to\C{I}^+$ and $\C{Y}^x_{w,\br}\to\C{I}_{w,\br}^+$ are finitely presented. Thus $\C{I}_{w,\br}^+,   \C{Y}^x$ and $\C{Y}^x_{w,\br}$ are schemes admitting placid presentations (by~\re{glplsch}(b)).


\end{Emp}

\begin{Emp} \label{E:bx+}
{\bf Notation.}
(a) Set
$I(x):=I\cap xI x^{-1}\subset\C{L}G$ and $\C{I}(x)^+:=\C{I}^+\cap\Ad_{x}(\C{I}^+)\subset\C{L}\g$.
Note that $\C{I}(x)^+$ was denoted by $\C{I}^+_{1,x}$ in~\re{affstvar}(e).

(b) Note that $\C{I}(x)^+$ is a scheme admitting a placid presentation, and $\C{I}(x)^+_{w,\br}\subset\C{I}(x)$ is finitely presented locally closed subscheme. Then $\C{I}(x)^+_{w,\br}$ admits a placid presentation (by~\re{glplsch}(b)), and we can consider the codimension function
\begin{equation} \label{Eq:bx}
\un{b}(x)^+_{w,\br}:=\un{\codim}_{\C{I}(x)^+}(\C{I}(x)^+_{w,\br})\colon \C{I}(x)^+_{w,\br}\to\B{Z}.
\end{equation}

(c) Note that $I(x)$ is a strongly pro-smooth group scheme. Since $\C{I}(x)^+$ and $\C{I}(x)^+_{w,\br}$ are $\Ad I(x)$-equivariant,
we can form quotient stacks $[I(x)\bs\C{I}(x)^+_{w,\br}]$ and $[I(x)\bs\C{I}(x)^+]$, both of which are placid (see~\re{basic example}). Using ~\re{bex2},
the codimension function $\un{b}(x)^+_{w,\br}$ of \form{bx} is induced by the codimension function
$\un{\codim}_{[I(x)\bs \C{I}(x)^+]}([I(x)\bs \C{I}(x)^+_{w,\br}])$,
which we also denote by $\un{b}(x)^+_{w,\br}$.
\end{Emp}

\begin{Emp}
{\bf Remark.}
If $x\in W$ is the unit element, then $\C{I}(x)^+=\C{I}^+$. In this case,
by~\cite[3.4.4(a) and Corollary 3.4.9]{BKV},
the function $\un{b}(x)^+_{w,\br}:=\un{\codim}_{\C{I}^+}(\C{I}^+_{w,\br})$ is the constant function
with value $\codim_{\C{L}^+(\fc)_\tn}(\fc_{w,\br})=\codim_{\C{L}^+(\fc)}(\fc_{w,\br})-r$, that was denoted by
$b^+_{w,\br}$ in~\cite[3.4.4(d)]{BKV}. Here $r=\dim\fc$ is the rank of $G$.
\end{Emp}


\begin{Lem} \label{L:basic}
(a) We have $\un{b}(x)^+_{w,\br}=0$ if $\pi(x)=[w,\br]$, and $\un{b}(x)^+_{w,\br}>0$ otherwise.

(b) We have natural isomorphisms $[I\bs\C{Y}^x]\simeq [I(x)\bs \C{I}(x)^+]$ and  $[I\bs\C{Y}_{w,\br}^x]\simeq [I(x)\bs \C{I}(x)^+_{w,\br}]$.



(c) The projection $\C{Y}^x\to\C{I}^+$ is  affine finitely presented, and $\un{\dim}(\C{Y}^x/\C{I}^+)=0$.
\end{Lem}

\begin{proof}
(a) By definition, $\pi(x)=[w,\br]$ is the unique class such that the GKM stratum
$\C{I}(x)^+_{w,\br}\subset\C{I}(x)^+$ is open dense. This implies the assertion.

(b) By definition,  $\C{Y}^x$ is an $I$-invariant closed subscheme of $\CN\times \Fl^x$ consisting of points $(\gm,[g])$ such that
$\gm\in\C{I}^+_g\cap\C{I}^+$, where $I$ acts by the formula $h(\gm,g)=(\Ad_h(\gm),hg)$. Since $I$ acts transitively in $\Fl^x$ and $I(x)\subset I$ is the stabilizer of $x\in\Fl$, the isomorphism $[I\bs\C{Y}^x]\simeq [I(x)\bs \C{I}(x)^+]$ follows.
The second isomorphism follows from the first by taking preimages of $\fc_{w,\br}\subset\C{L}^+(\fc)$.

(c) Taking the quotient by $I$, it suffices to show that the projection
$[I\bs\C{Y}^x]\to[I\bs\C{I}^+]$ is affine finitely presented of constant dimension zero (compare~\re{bex2}).
Using~(b), this projection can be identified with the composition
$[I(x)\bs\C{I}(x)^+]\to [I(x)\bs\C{I}^+]\to [I\bs\C{I}^+]$. Since $\C{I}(x)^+\subset\C{I}^+$ is closed finitely presented subscheme, while
$I/I(x)\simeq\C{I}^+/\C{I}(x)^+$ is non-canonically isomorphic to an affine space, the assertion follows.
\end{proof}

\begin{Emp} \label{E:form}
{\bf Notation.} (a) As in~\cite[3.4.1]{BKV}, we set $d_{\br}:=\sum_{\alpha\in R}\br(\alpha)$,
$c_w:=\dim\ft-\dim\ft^w$, where $\ft^w\subset\ft$ denotes the space of $w$-invariants, and  $\dt_{w,\br}:=\frac{1}{2}(d_\br-c_w)$.

(b) Note that $\ft_{w,\br}\subset\C{L}^+(\ft_w)_\tn$ is a connected strongly pro-smooth finitely presented locally closed subscheme
(see~\cite[3.3.3]{BKV}) of constant codimension (see~\cite[Lemma 2.2.10]{BKV}). As in~\cite[3.4.4(d)]{BKV}, we set
$a^+_{w,\br}:=\codim_{\C{L}^+(\ft_w)_\tn}(\ft_{w,\br})$.

(c) Using~\rl{basic}(b), we have a natural projection
$\C{Y}_{w,\br}^x\to[I\bs\C{Y}_{w,\br}^x]\simeq [I(x)\bs \C{I}(x)^+_{w,\br}]$. Denote by $\un{b}(x)^+_{w,r}|_{\C{Y}^x_{w,\br}}$ the pullback of the codimension function $\un{b}(x)^+_{w,r}:[I(x)\bs \C{I}(x)^+_{w,\br}]\to \B{Z}$ (see~\re{bx+}(c)).
\end{Emp}

\begin{Emp} \label{E:remmin}
{\bf Remark.} Since $\C{L}^+(\ft_w)_\tn$ is irreducible, it has a unique open dense stratum $\ft_{w,\br}$, while all other strata are of positive
codimension. Therefore a class $[w,\br]\in\fP$ is {\em minimal} if and only if $a^+_{w,\br}=0$.
\end{Emp}

\begin{Lem} \label{L:dim}
We have an equality
\[
\un{\dim}(\C{Y}^x_{w,\br}/\C{I}^+_{w,\br})=\dt_{w,\br}+a^+_{w,\br}-(\un{b}(x)^+_{w,r}|_{\C{Y}^x_{w,\br}}).
\]
\end{Lem}

\begin{proof}
By the additivity of the dimension function (see~\re{propdim}(a)), we have
\[
\un{\dim}(\C{Y}^x_{w,\br}/\C{I}^+)=\un{\dim}(\C{Y}^x_{w,\br}/\C{I}^+_{w,\br})-(\un{\codim}_{\C{I}^+}(\C{I}^+_{w,\br})|_{\C{Y}^x_{w,\br}})
\]
and
\[
\un{\dim}(\C{Y}^x_{w,\br}/\C{I}^+)=
(\un{\dim}(\C{Y}^x/\C{I}^+)|_{\C{Y}^x_{w,\br}})-\un{\codim}_{\C{Y}^x}(\C{Y}^x_{w,\br}).
\]
Thus
\[
\un{\dim}(\C{Y}^x_{w,\br}/\C{I}^+_{w,\br})=(\un{\codim}_{\C{I}^+}(\C{I}^+_{w,\br})|_{\C{Y}^x_{w,\br}})+
(\un{\dim}(\C{Y}^x/\C{I}^+)|_{\C{Y}^x_{w,\br}})-\un{\codim}_{\C{Y}^x}(\C{Y}^x_{w,\br}).
\]
Note that it follows from~\cite[Corollaries 3.4.5 and 3.4.9]{BKV} that the closed
subscheme $\C{I}^+_{w,\br}\subset \C{I}^+$ is of constant codimension $\codim_{\C{I}^+}(\C{I}^+_{w,\br})=\dt_{w,\br}+a^+_{w,\br}$.

Since $\un{\dim}(\C{Y}^x/\C{I}^+)=0$ by~\rl{basic}(c), it suffices to show the equality
\[
\un{\codim}_{\C{Y}^x}(\C{Y}^x_{w,\br})=(\un{b}(x)^+_{w,r}|_{\C{Y}^x_{w,\br}}),
\]
which follows from the observation~\re{bex2} and~\rl{basic}(b).
\end{proof}

Now we are ready to show the first part of Lusztig conjecture.

\begin{Thm} \label{T:min}
For every $x\in W$, the class $\pi(x)=[w,\br]$ is minimal.
\end{Thm}

\begin{proof}
By the formula of Bezrukavnikov--Kazhdan--Lusztig (see~\cite{Be}), all fibers of the projection $\C{Y}_{w,\br}\to\C{I}^+_{w,\br}$ are of dimension
$\dt_{w,\br}$. Therefore all fibers of $\C{Y}^x_{w,\br}\to\C{I}^+_{w,\br}$ are of dimension at most $\dt_{w,\br}$, hence by~\rl{dimfib}(a) we have
$\un{\dim}(\C{Y}^x_{w,\br}/\C{I}^+_{w,\br})\leq \dt_{w,\br}$. It now follows from~\rl{dim} that
$a^+_{w,\br}\leq \un{b}(x)^+_{w,\br}$. Next, since $\pi(x)=[w,\br]$, we conclude by~\rl{basic}(a) that $\un{b}(x)^+_{w,\br}=0$.
Thus $a^+_{w,\br}=0$, hence the class $[w,\br]$ is minimal by Remark~\re{remmin}.
\end{proof}

\section{Proof of Conjecture ~\re{conj}(b) for generic elements}

We continue to fix $x\in W$ and $[w,\br]\in\fP$.



\begin{Emp} \label{E:cons}
{\bf Notation.}
(a) Recall (see~\cite[4.1.5]{BKV}) that element $w\in W_{\fin}$ gives rise to a maximal torus $T_w\subset G_{{\mathsf k}(\!(t)\!)}$,
hence  to an ind-subgroup scheme $\C{L}(T_w)\subset\C{L}G$, both defined uniquely up to conjugacy. Moreover, we have a natural $\C{L}(T_w)$-equivariant embedding $\ft_{w,\br}\hra\CN_{w,\br}$, defined uniquely up to conjugacy, where $\C{L}(T_w)$ acts trivially on $\ft_{w,\br}$.

(b)  We set $\wt{\CN}_{\ft,w,\br}:=\ft_{w,\br}\times_{\CN_{w,\br}}\wt{\CN}_{w,\br}$, and
$\St^x_{\ft,w,\br}:=\ft_{w,\br}\times_{\CN_{w,\br}}\St^x_{w,\br}$. Both $\wt{\CN}_{\ft,w,\br}$ and
$\St^x_{\ft,w,\br}$ are ind-schemes over $\ft_{w,\br}$.

(c) Consider the composition $\pr\colon \St^x\hra \St=\wt{\CN}\times_{\CN}\wt{\CN}\overset{\pr_1}{\lra} \wt{\CN}$. It is $\C{L}G$-equivariant, and therefore induces an $\C{L}G$-equivariant projection $\pr\colon \St^x_{w,\br}\to \wt{\CN}_{w,\br}$, hence an $\C{L}(T_w)$-equivariant projection
$\pr_{\ft}\colon \St^x_{\ft,w,\br}\to\wt{\CN}_{\ft,w,\br}$.

(d) Let $\La_w:=X_*(T_w)$ be the group of cocharacters of $T_w$, defined over ${\mathsf k}(\!(t)\!)$. It is a finitely generated free abelian group, and we have  natural embedding $\La_w\hra \C{L}(T_w),\ \la\mapsto\la(t)$. In particular, the projection $\pr_{\ft}\colon \St^x_{\ft,w,\br}\to \wt{\CN}_{\ft,w,\br}$ from (c) is $\La_w$-equivariant.
\end{Emp}

\begin{Lem} \label{L:propSt}
(a) We have natural isomorphisms  $[\C{L}G\bs\wt{\CN}_{w,\br}]\simeq [I\bs \C{I}_{w,\br}^+]$ and
$[\C{L}G\bs\St^x_{w,\br}]\simeq [I\bs \C{Y}_{w,\br}^x]$.

(b) The quotient stacks $[\C{L}G\bs\wt{\CN}_{w,\br}]$ and $[\C{L}G\bs\St^x_{w,\br}]$
are placid, and the projection
\[
[\pr]\colon [\C{L}G\bs\St^x_{w,\br}]\to[\C{L}G\bs\wt{\CN}_{w,\br}]
\]
is an affine and finitely presented.

(c) The reduced ind-schemes $(\wt{\CN}_{\ft,w,\br})_{\red}$ and
$(\St^x_{\ft,w,\br})_{\red}$ are placid schemes, locally finitely presented over $\ft_{w,\br}$, while the projection
$\pr_{\ft,\red}\colon (\St^x_{\ft,w,\br})_{\red}\to (\wt{\CN}_{\ft,w,\br})_{\red}$ is affine and finitely presented.

(d) The quotients $[\La_w\bs (\wt{\CN}_{\ft,w,\br})_{\red}]$ and $[\La_w\bs (\St^x_{\ft,w,\br})_{\red}]$ are placid algebraic spaces, finitely presented over $\ft_{w,\br}$.
\end{Lem}

\begin{proof}
(a) follows from the observation that both $\pr\colon \St^x\to\wt{\CN}$ from~\re{cons}(c) and the projection $\wt{\CN}\to\Fl$ are
$\C{L}G$-equivariant, and the fiber of $\pr$ over $[1]\in \Fl$ is the projection
$\C{Y}^x\to\C{I}^+$.

(b) Since the projection $\C{Y}^x\to\C{I}^+$ and its pullback $\C{Y}_{w,\br}^x\to\C{I}_{w,\br}^+$ are affine and finitely presented
(by~\rl{basic}(c)), all assertions follows from ~\re{basic example} and the statement and the proof of (a).

(c) It follows from~\cite[Theorem 4.3.3]{BKV}, that $(\wt{\CN}_{\ft,w,\br})_{\red}$ is a scheme, locally finitely presented over $\ft_{w,\br}$.
Since $\ft_{w,\br}$ is placid (compare~\re{form}), we conclude that $(\wt{\CN}_{\ft,w,\br})_{\red}$ is a placid scheme by~\re{plalgsp}(a).
Next, using identifications
\begin{equation} \label{Eq:fiberproduct}
\wt{\CN}_{\ft,w,\br}\simeq\ft_{w,\br}\times_{[\C{L}G\bs \CN_{w,\br}]}[\C{L}G\bs\wt{\CN}_{w,\br}]\text{ and }
\St^x_{\ft,w,\br}\simeq\ft_{w,\br}\times_{[\C{L}G\bs\CN_{w,\br}]}[\C{L}G\bs\St^x_{w,\br}],
\end{equation}
we deduce from (b) that the projection $\pr_{\ft}\colon \St^x_{\ft,w,\br}\to \wt{\CN}_{\ft,w,\br}$ is affine and finitely presented. Therefore the remaining assertions follow from \re{red}(a)
and~\re{plalgsp}(a).

(d) By~\cite[Theorem 4.3.3]{BKV}, the quotient $[\La_w\bs(\wt{\C{N}}_{\ft,w,\br})_{\red}]$ is an algebraic space, finitely presented over $\ft_{w,\br}$. Therefore is placid by~\re{plalgsp}(a). Moreover, we conclude from (c) that the projection
$[\La_w\bs (\St^x_{\ft,w,\br})_{\red}]\to[\La_w\bs (\wt{\CN}_{\ft,w,\br})_{\red}]$ is affine and finitely presented, which implies that
$[\La_w\bs (\St^x_{\ft,w,\br})_{\red}]$ is a placid algebraic space, finitely presented over $\ft_{w,\br}$.
\end{proof}

\begin{Emp} \label{E:remschemes}
{\bf Remark.}
By~\cite[Corollary 4.3.4(a)]{BKV}, there exists a subgroup of finite
index $\La'_w\subset\La_w$ such that the quotient
$[\La'_w\bs(\wt{\C{N}}_{\ft,w,\br})_{\red}]$ is a scheme finitely presented over $\ft_{w,\br}$.
Thus, using~\rl{propSt}(c), one deduces that the quotient
$[\La'_w\bs(\St^x_{\ft,w,\br})_{\red}]$ is a scheme finitely presented over $\ft_{w,\br}$ as well.
In particular, for the purpose of this work we could restrict ourselves to schemes instead of algebraic spaces.
\end{Emp}

\begin{Emp} \label{E:notproj}
{\bf Notation.}
(a) Composing isomorphisms of~\rl{propSt}(a) and~\rl{basic}(b), we get an isomorphism
$[\C{L}G\bs\St^x_{w,\br}]\simeq  [I\bs\C{Y}^x_{w,\br}]\simeq[I(x)\bs \C{I}(x)^+_{w,\br}]$.

(b) By (a), we have a natural projection
$\St^x_{\ft,w,\br}\to\St^x_{w,\br}\to [\C{L}G\bs\St^x_{w,\br}]\simeq [I(x)\bs \C{I}(x)^+_{w,\br}]$.
Therefore we can consider the pullback $\un{b}(x)^+_{w,\br}|_{(\St^x_{\ft,w,\br})_{\red}}$ (see~\re{bx+}(c)).

(c) For every $\wt{\gm}\in \St^x_{w,\br}$, we denote its image in $[I(x)\bs \C{I}(x)^+_{w,\br}]$ by $[\wt{\gm}]$.
\end{Emp}

The following assertion is the main technical result of this work.

\begin{Prop} \label{P:dimform}
We have an equality
\[
\un{\dim}((\St^x_{\ft,w,\br})_{\red}/\ft_{w,\br})=2\dt_{w,\br}+a^+_{w,\br}-(\un{b}(x)^+_{w,\br}|_{(\St^x_{\ft,w,\br})_{\red}}).
\]
\end{Prop}

Before giving the proof of~\rp{dimform}, we are going to explain how Lusztig's conjecture~\re{conj}(b) for
generic elements follows from it.

\begin{Cor} \label{C:generic1}
(a) There exists an open dense subscheme ${}^x\ft_{w,\br}\subset \ft_{w,\br}$ such that the function $\gm\mapsto\dim\St_{\gm}^x$ is
constant on ${}^x\ft_{w,\br}$, and for every $\gm\in {}^x\ft_{w,\br}$ and
$\wt{\gm}\in \St^x_{\gm}$ we have an equality
\[
\dim_{\wt{\gm}}\St^x_{\gm}=2\dt_{w,\br}+a^+_{w,\br}-\un{b}(x)^+_{w,\br}([\wt{\gm}]).
\]

(b) If $[w,\br]$ is minimal and $\pi(x)=[w,\br]$, then for every $\gm\in{}^x\ft_{w,\br}$, the fiber $\St_{\gm}^x$ is non-empty and equidimensional of dimension $2\dt_{w,\br}$.
\end{Cor}

\begin{proof}
(a) Recall that the projection $f\colon [\La_w\bs(\St^x_{\ft,w,\br})_{\red}]\to\ft_{w,\br}$ is
finitely presented by~\rl{propSt}(d), and let ${}^x\ft_{w,\br}\subset \ft_{w,\br}$ be the largest open
dense subset satisfying the condition of~\rl{dimfib}(c) for $f$. For every $\wt{\gm}\in \St^x_{\gm}$, let $\wt{\gm}'$ be the projection of $\wt{\gm}$ to  $[\La_w\bs(\St^x_{\ft,w,\br})_{\red}]$.

Then we have a sequence of equalities
\[
\dim_{\wt{\gm}}\St^x_{\gm}=\dim_{\wt{\gm}'}f^{-1}(\gm)=\un{\dim}_f(\wt{\gm}')=\un{\dim}((\St^x_{\ft,w,\br})_{\red}/\ft_{w,\br})(\wt{\gm})=
2\dt_{w,\br}+a^+_{w,\br}-\un{b}(x)^+_{w,\br}([\wt{\gm}]),
\]
where

\quad $\bullet$ the first equality follows from the identification
$[\La_w\bs\St_{\gm}^x]\simeq f^{-1}(\gm)_{\red}$;

\quad $\bullet$ the second one follows from the assumption on ${}^x\ft_{w,\br}$;

\quad $\bullet$ the third equality is clear;

\quad $\bullet$ the last one follows by~\rp{dimform}.

(b) If $\pi(x)=[w,\br]$ is minimal, then $a^+_{w,\br}=0$ (by~Remark~\re{remmin}) and
$\un{b}(x)^+_{w,\br}=0$ (by~\rl{basic}(a)). In this case, assertion~(a) implies that for every
$\gm\in{}^x\ft_{w,\br}$, the fiber $\St_{\gm}^x$ is either equidimensional of dimension
$2\dt_{w,\br}$ or empty. Thus it remains to show that each  $\St_{\gm}^x$ is non-empty. Equivalently,
in the notation of the proof of~(a) it remains to show that for each $\gm\in{}^x\ft_{w,\br}$,
the fiber $f^{-1}(\gm)_{\red}=[\La_w\bs\St_{\gm}^x]$ is non-empty.

Since $\ft_{w,\br}$ is irreducible, it remains to show that
$\dim_{\wt{\gm}'}f^{-1}(\gm)\leq\un{\dim}_{f}(\wt{\gm}')$ for every
$\gm\in\ft_{w,\br}$ and $\wt{\gm}'\in f^{-1}(\gm)$ (by~\rl{dimfib}(a),(d)). Arguing as in~(a),
it suffices to show that
\[
\dim_{\wt{\gm}}\St^x_{\gm}\leq\un{\dim}((\St^x_{\ft,w,\br})_{\red}/\ft_{w,\br})(\wt{\gm}).
\]

But this follows from the fact that the RHS equals $2\dt_{w,\br}$ by~\rp{dimform}, and the LHS is at most
$\dim\St_{\gm}=2\dt_{w,\br}$.
\end{proof}

\begin{Emp}
{\bf Notation.}
(a) Let ${}^x\ft_{w,\br}\subset \ft_{w,\br}$ be the largest open subscheme satisfying the property
of~\rco{generic1}(a). Since
the map $\C{L}\chi\colon \ft_{w,\br}\to\fc_{w,\br}$ is finite \'etale (see~\cite[3.3.4(c)]{BKV}), the image ${}^x\fc_{w,\br}:=\C{L}\chi({}^x\ft_{w,\br})$ is an open dense subscheme of $\fc_{w,\br}$.

(b) We set  ${}^x\CN_{w,\br}:=\C{L}\chi^{-1}({}^x\fc_{w,\br})\subset \CN_{w,\br}$.
\end{Emp}

\begin{Cor} \label{C:generic}
(a) For every $\gm\in {}^x\CN_{w,\br}$ and $\wt{\gm}\in \St^x_{\gm}$ we have an equality
\[
\dim_{\wt{\gm}}\St^x_{\gm}=2\dt_{w,\br}+a^+_{w,\br}-\un{b}(x)^+_{w,\br}([\wt{\gm}]).
\]

(b) If $[w,\br]$ is minimal, and $\pi(x)=[w,\br]$, then the fiber $\St_{\gm}^x$ is non-empty and equidimensional of dimension
$2\dt_{w,\br}$.
\end{Cor}

\begin{proof}
  By construction, for every $\gm\in {}^x\CN_{w,\br}$ there exists $\gm'\in {}^x\ft_{w,\br}$ such
  that $\C{L}\chi(\gm')=\C{L}\chi(\gm)$. Thus there exists $g\in \C{L}G$ such that $\Ad_g(\gm)=\gm'$.
  Then $g$ induces an isomorphism $g\colon \St^{x}_{\gm}\isom \St^{x}_{\gm'}$, thus
  $\dim_{\wt{\gm}}\St^x_{\gm}=\dim_{g(\wt{\gm})}\St^x_{\gm'}$. Since $[g(\wt{\gm})]=[\wt{\gm}]$ (see~\re{notproj}) and
  the corresponding assertions for $\St^x_{\gm'}$ were shown in~\rco{generic1}, the assertion for
  $\St^x_{\gm}$ follows.
\end{proof}

As a particular case, we deduce Lusztig conjecture~\re{conj}(b) for generic elements.

\begin{Thm} \label{T:lusztig}
Assume that the class $[w,\br]\in\fP$ is minimal. Then for every $\gm\in {}^x\CN_{w,\br}$, the fiber $\St_{\gm}^x$ satisfies

\quad $\bullet$ $\dim \St_{\gm}^x=2\dt_{w,\br}$, if $\pi(x)=[w,\br]$.

\quad $\bullet$ $\dim \St_{\gm}^x<2\dt_{w,\br}$,  if $\pi(x)\neq [w,\br]$.
\end{Thm}

\begin{proof}
  Since $[w,\br]$ is minimal, we have $a^+_{w,\br}=0$ (see~Remark~\re{remmin}). If
  $\pi(x)\neq[w,\br]$, we have $\un{b}(x)^+_{w,\br}>0$ (by~\rl{basic}(a)).
  Then~\rco{generic}(a) implies that $\dim_{\wt{\gm}} \St_{\gm}^x<2\dt_{w,\br}$ for every
  $\wt{\gm}\in\St_{\gm}^x$, thus $\dim \St_{\gm}^x<2\dt_{w,\br}$. The assertion for $\pi(x)=[w,\br]$
  follows from~\rco{generic}(b).
\end{proof}

For completeness, we now deduce Lusztig conjectures~\re{conj}(c),(d) from~\rt{lusztig}.

\begin{Cor} \label{C:lusztig}
(a) For every $w\in W_{\fin}$, the preimage $\ov{\pi}^{-1}([w])$ is non-empty.

(b) Moreover, if $G$ is semisimple,
then $\ov{\pi}^{-1}([w])$ if finite if and only if $w$ is elliptic.
\end{Cor}

\begin{proof}
Choose the unique function $\br:R\to\B{Q}_{>0}$ such that the GKM pair $(w,\br)$ is minimal (see~\re{min}(c)). Then, by~\rt{min}, we have an equality $\ov{\pi}^{-1}([w])=\pi^{-1}([w,\br])$. Next we recall that the GKM stratum $\fc_{w,\br}$ is irreducible (see~\cite[3.3.4]{BKV}),
and choose a geometric point $\ov{\gm}\in\CN_{w,\br}$ whose image
$\C{L}\chi(\ov{\gm})\in \fc_{w,\br}$ is supported at a generic point. Then $\ov{\gm}\in {}^x\CN_{w,\br}$ for every $x\in W$.

Since $\St_{\ov{\gm}}$ is equidimensional of dimension $2\dt_{w,\br}$, it follows from~\rt{lusztig}
that the preimage $\ov{\pi}^{-1}([w])=\pi^{-1}([w,\br])$ consists of all $x\in W$ such that $\St^x_{\ov{\gm}}$ contains a generic point (of some irreducible component) of $\St_{\ov{\gm}}$. From this both assertions follow:

(a) Let $\wt{\gm}\in\St_{\ov{\gm}}$ be a generic point of  $\St_{\ov{\gm}}$. Then $\wt{\gm}\in \St^x_{\ov{\gm}}$ for some $x\in W$, which by the observation above implies that $x\in \ov{\pi}^{-1}([w])$.

(b) Let now $G$ be semisimple, and assume that $w$ is elliptic. Then the affine Springer fiber $\Fl_{\ov{\gm}}$, and hence also the affine Steinberg fiber   $\St_{\ov{\gm}}=\Fl_{\ov{\gm}}\times\Fl_{\ov{\gm}}$ has finitely many generic points, which implies that
$\ov{\pi}^{-1}([w])$ is finite.

Assume now that $w$ is not elliptic, thus $\La_w\neq 0$. Choose a generic point $\wt{\gm}=(\ov{\gm},g_1,g_2)\in\St_{\ov{\gm}}$ of
$\St_{\ov{\gm}}$. Then for every $\la\in \La_w$, the translate $\wt{\gm}_{\la}:=(1,\la)(\wt{\gm})=(\ov{\gm},g_1,\la g_2)$ is a generic point of $\St_{\ov{\gm}}$ as well, and $\wt{\gm}_{\la}\in \St^{x_{\la}}_{\ov{\gm}}$, where $x_{\la}$ be the class $[g_1^{-1}\la g_2]\in I\bs\C{L}G/I=W$.
Thus the assertion follows from the observation that the set $\{x_{\la}\}_{\la\in\La_w}\subset W$ is infinite.
\end{proof}

\section{Proof of \rp{dimform}}
\begin{Emp} \label{E:step1}
By the additivity of the dimension function (see \re{propdim}(a)), it suffices to show equalities
\begin{equation} \label{Eq:dim1}
\un{\dim}((\wt{\CN}_{\ft,w,\br})_{\red}/\ft_{w,\br})=\dt_{w,\br}
\end{equation}
\quad and
\begin{equation} \label{Eq:dim2}
\un{\dim}((\St^x_{\ft,w,\br})_{\red}/(\wt{\CN}_{\ft,w,\br})_{\red})=\dt_{w,\br}+a^+_{w,\br}-(\un{b}(x)^+_{w,\br}|_{(\St^x_{\ft,w,\br})_{\red}}).
\end{equation}

Applying~\rl{dimfib}(b) to the projection $f\colon [\La_w\bs(\wt{\CN}_{\ft,w,\br})_{\red}]\to\ft_{w,\br}$,
 equality~\form{dim1} will follow if we show that $f$ is open, and all the fibers of $f$ are
equidimensional of dimension $\dt_{w,\br}$. While the first assertion was proved
in~\cite[Corollary 4.3.4(c)]{BKV} (compare the proof of~\rp{flat} below), the second one follows from the fact for every
$\gm\in \ft_{w,\br}$ we have $f^{-1}(\gm)_{\red}\simeq [\Fl_{\gm}/\La_w]$ and formula~\cite{Be}
for the dimension of affine Springer fibers.
\end{Emp}

\begin{Emp} \label{E:step2}
To show equality \form{dim2}, notice that we have a Cartesian diagram
\begin{equation} \label{Eq:basic CD}
\begin{CD}
\St_{\ft,w,\br}^x @>\psi''_{w,\br}>> [\C{L}G\bs\St_{w,\br}^x] @>\sim>>  [I\bs\C{Y}_{w,\br}^x] @>\sim>>  [I(x)\bs\C{I}(x)^+_{w,\br}]\\
@V\pr_{\ft}VV @V[\pr]VV @V p^x VV @VVV\\
\wt{\CN}_{\ft,w,\br} @>\psi'_{w,\br}>>  [\C{L}G\bs\wt{\CN}_{w,\br}] @>\sim>> [I\bs\C{I}^+_{w,\br}] @= [I\bs\C{I}^+_{w,\br}],
\end{CD}
\end{equation}
where the middle horizontal isomorphisms are those of~\rl{propSt}(a), and right top horizontal isomorphism is that of~\rl{basic}(b).  Using~\rl{dim}, we therefore conclude that
\begin{equation} \label{Eq:dimpr1}
\un{\dim}_{[\pr]}=\un{\dim}_{p^x}= \dt_{w,\br}+a^+_{w,\br}-\un{b}(x)^+_{w,\br}.
\end{equation}
\end{Emp}

\begin{Emp} \label{E:step3}
Consider commutative diagram
\begin{equation} \label{Eq:basic CD2}
\begin{CD}
(\St_{\ft,w,\br}^x)_{\red} @>\psi''_{w,\br,\red}>> [\C{L}G\bs\St_{w,\br}^x]_{\red}\\
@V\pr_{\ft,\red}VV @V[\pr]_{\red} VV \\
(\wt{\CN}_{\ft,w,\br})_{\red} @>\psi'_{w,\br,\red}>> [\C{L}G\bs\wt{\CN}_{w,\br}]_{\red}\\
@VVV @VVV\\
\ft_{w,\br} @>\psi_{w,\br}>> [\C{L}G\bs \CN_{w,\br}]_{\red}.
\end{CD}
\end{equation}
Combining \form{dimpr1} with \re{propdim}(b), it suffices to show the equality  $\un{\dim}_{\pr_{\ft,\red}}=(\psi''_{w,\br,\red})^*(\un{\dim}_{[\pr]_{\red}})$.
\end{Emp}

\begin{Emp} \label{E:step4}
Since $\ft_{w,\br}$ is reduced, it follows from identification $(\C{X}_{\red}\times_{\C{Y}_{\red}}\C{Z}_{\red})_{\red}\simeq (\C{X}\times_{\C{Y}}\C{Z})_{\red}$ (see~\cite[1.4.1(e)]{BKV}) and identities \form{fiberproduct} that the bottom inner square and the exterior
square of \form{basic CD2} are Cartesian. Therefore the top inner square of \form{basic CD2}
is Cartesian as well. Next, by~\cite[Corollary 4.1.12]{BKV}, the projection $\psi_{w,\br}:\ft_{w,\br}\to [\C{L}G\bs \CN_{w,\br}]_{\red}$ is a placid atlas. Therefore its pullbacks $\psi'_{w,\br,\red}$ and $\psi''_{w,\br,\red}$ are placid atlases as well, thus the assertion of \re{step3}
follows from the definition of the dimension function in \re{dimplst}(c).
\end{Emp}
%
%

\section{Flatness conjecture}

\begin{Conj} \label{C:flat}
For every $x\in W$ and $[w,\br]\in\fP$, we have either $\C{I}(x)^+_{w,\br}=\emptyset$, or the projection
$\C{I}(x)^+_{w,\br}\to\fc_{w,\br}$ is faithfully flat.
\end{Conj}

\begin{Emp}
{\bf Example.} Note that the projection $\C{I}^+\to\C{L}^+(\fc)$ is flat (see ~\cite[Corollary 3.4.8]{BKV}), and it is known to be surjective at least when the characteristic of $\mathsf k$ is sufficiently large. Therefore Conjecture~\ref{C:flat} holds for
$x=1$. Moreover, in this case, $\C{I}(x)^+_{w,\br}\neq\emptyset$ for all $[w,\br]\in\fP$.
\end{Emp}

The following result shows that Conjecture~\ref{C:flat} implies the full
Lusztig conjecture~\re{conj}(b).

\begin{Prop} \label{P:flat}
Assume that Conjecture~\ref{C:flat} holds for a triple $(x,w,\br)$.

(a) Then for every $\gm\in\CN_{w,\br}$ and $\wt{\gm}\in\St_{\gm}^x$, we have
\[
\dim_{\wt{\gm}}\St_{\gm}^x=2\dt_{w,\br}+a^+_{w,\br}-\un{b}(x)^+_{w,\br}([\wt{\gm}]).
\]

(b) Assume that $[w,\br]\in\fP$ is minimal. Then for every $\gm\in\CN_{w,\br}$, we have
\[
\dim \St_{\gm}^x=2\dt_{w,\br},\ \operatorname{if}\ \pi(x)=[w,\br];\ \operatorname{and}\
\dim \St_{\gm}^x<2\dt_{w,\br},\ \operatorname{if}\ \pi(x)\neq [w,\br].
\]
\end{Prop}

\begin{proof}
As in~\rt{lusztig}, assertion (b) follows from (a). So it remains to show assertion (a).
If $\C{I}(x)^+_{w,\br}$ is empty, we get $[\C{L}G\bs\St^x_{w,\br}]\simeq [I(x)\bs \C{I}(x)^+_{w,\br}]=\emptyset$
(see~\re{notproj}(a)), hence $\St^x_{\ft,w,\br}$ is empty. Assume from now on that $\C{I}(x)^+_{w,\br}\to\fc_{w,\br}$ is faithfully flat.

Using~\rl{dimfib}(b), and arguing as in~\rco{generic1}(a) and~\rco{generic}(a), it suffices to show that the projection
$[\La_w\bs(\St^x_{\ft,w,\br})_{\red}]\to \ft_{w,\br}$ or, equivalently, projection $p\colon(\St^x_{\ft,w,\br})_{\red}\to \ft_{w,\br}$ is open and surjective.

To prove the assertion, we basically repeat the argument of~\cite[Proposition 4.3.1]{BKV}.
Since this proof uses a lot of terminology, which was not discussed in this work, we provide a direct
argument instead.

Choose a GKM pair $(w,\br)$ in the class $[w,\br]$, and let $W_{w,\br}\subset W_{\fin}$ be the stabilizer of $(w,\br)$. Then $W_{w,\br}$ acts freely on $\ft_{w,\br}$ and induces an isomorphism
$[W_{w,\br}\bs\ft_{w,\br}]\simeq\fc_{w,\br}$ (see~\cite[3.3.4(d)]{BKV}). Moreover, the projection
$\ft_{w,\br}\to[\C{L}G\bs\CN_{w,\br}]$ is $W_{w,\br}$-invariant. Therefore the projection $p$ is $W_{w,\br}$-equivariant,
so it suffices to show that the composition $\ov{p}\colon(\St^x_{\ft,w,\br})_{\red}\overset{p}{\to} \ft_{w,\br}\to\fc_{w,\br}$ is universally open and surjective.

Consider commutative diagram
\begin{equation} \label{Eq:last CD}
\begin{CD}
\C{X} @>(2)>> (\St_{\ft,w,\br})_{\red} @>\ov{p}>> \fc_{w,\br} \\
@V(4)VV @V(3)VV @|\\
\C{I}(x)_{w,\br}^+ @>(1)>> [I(x)\bs\C{I}(x)^+_{w,\br}] @>>> \fc_{w,\br},
\end{CD}
\end{equation}
where the left square is Cartesian, and morphism $(1)$ is the projection, and morphism $(3)$ is induced by the top horizontal arrow
of~\form{basic CD}.

As it was mentioned in \re{step4}, the map $(3)_{\red}\colon (\St_{\ft,w,\br})_{\red}\to [I(x)\bs\C{I}(x)^+_{w,\br,\red}]$
is a placid atlas. Therefore the map $(3)$ is surjective, so surjectivity of $\ov{p}$ follows from that of $\C{I}(x)_{w,\br}^+\to\fc_{w,\br}$.

Next, since $\ov{p}$ is locally finitely presented, in order to show that it is universally open, it suffices to show that
generalizations lift along every base change of $\ov{p}$ (see~\cite[Tag 01U1]{St}).

Since $(2)$ is a pullback of $(1)$, it is an $I(x)$-torsor. In particular, $\C{X}$ is a scheme, and the map $(2)$ is surjective.
Thus, using the commutativity of \form{last CD}, it suffices to show that generalizations lift along every base change of maps $(4)$ and
$\C{I}(x)_{w,\br}^+\to\fc_{w,\br}$.

Since $(4)_{\red}\colon \C{X}\to  \C{I}(x)^+_{w,\br,\red}$ is a pullback of $(3)_{\red}$, it is a smooth atlas. Thus $(4)_{\red}$ is flat, so both assertions follow from the fact that generalizations lift along flat morphisms of schemes (see~\cite[Tag 03HV]{St}).
\end{proof}

%


\begin{thebibliography}{99}


\bibitem[B]{Be} R.~Bezrukavnikov, {\em Dimension of the fixed point set on affine flag
  manifolds}, Math.\ Res.\ Letters {\bf 3} (1996), 185--189.


\bibitem[BM]{BM}
W.~Borho, R.~MacPherson, {\em Partial resolutions of nilpotent varieties},
Ast\'erisque, {\bf 101-102} (1983), 23--74.

\bibitem[BKV]{BKV} A.~Bouthier, D.~Kazhdan, Y.~Varshavsky, {\em Perverse sheaves
  on infinite-dimensional stacks, and affine Springer Theory}, arXiv:2003.01428.

\bibitem[BYY]{BYY} P.~Boixeda Alvarez, L.~Ying, G.~Yue, {\em Affine Springer fibers
  and the affine matrix ball construction for rectangular type nilpotents},
  arXiv:2007.14501.

\bibitem[EGA]{EGA} A.~Grothendieck, r\'edig\'es avec la collaboration de J.~Dieudonn\'e,
  {\em \'El\'ements de g\'eom\'etrie alg\'ebrique. IV.
 \'Etude locale des sch\'emas ed des morphismes de sch\'emas (Troisi\`eme Partie)},
  Inst. Hautes \'Etudes Sci.\ Publ.\ Math.\ {\bf 28} (1966), 5--255.


\bibitem[GKM]{GKM} M.~Goresky, R.~Kottwitz, R.~MacPherson, {\em Codimension
  of root valuation strata}, Pure and Appl.\ Math.\ Quarterly {\bf 5} (2009), 1253--1310.

\bibitem[KL]{KL} D.~Kazhdan, G.~Lusztig, {\em Fixed point varieties on affine flag
  manifolds}, Israel J.\ of Math.\ {\bf 62} (1988), 129--168.

\bibitem[La]{La} G.~Lawton, {\em Two-sided cells in the affine Weyl group of type
  $\tilde{A}_{n-1}$}, J.\ Algebra {\bf 120} (1989), no.~1, 74--89.

\bibitem[Lu1]{L} G.~Lusztig, {\em The two-sided cells of the affine Weyl group of type
  $\widetilde{A}_n$}, Math.\ Sci.\ Res.\ Inst.\ Publ.\ {\bf 4}, Springer, New York
  (1985), 275--283.

\bibitem[Lu2]{Lu} G.~Lusztig, {\em Affine Weyl groups and conjugacy classes in Weyl groups},
  Transformation Groups {\bf 1} (1996), 83--97.

\bibitem[Stacks]{St} {\em Stacks project}, http://stacks.math.columbia.edu/.









\end{thebibliography}
\end{document}